\documentclass[12pt]{article}
\usepackage{amsmath,mathtools,dsfont,amssymb}
\usepackage{xcolor}
\usepackage{latexsym}
\usepackage{amsthm}

\newcommand{\A}{\mathcal{A}}
\newcommand{\U}{\mathcal{U}}

\newcommand{\PA}{\mathcal{P}}
\newcommand{\SA}{\mathcal{S}}

\newcommand{\M}{\widetilde{M}}

\newcommand{\Q}{\mathcal{Q}}

\newcommand{\B}{\mathcal{B}}

\newcommand{\RR}{\mathbb{R}} 
\newcommand{\Exp}{\mathbb{E}} 

\newcommand{\X}{\widetilde{X}}

\newtheorem{theorem}{Theorem}

\begin{document}

\begin{center}
\textbf{ RISK-SENSITIVE CONTROL, SINGLE CONTROLLER\\ \ \\  GAMES AND LINEAR PROGRAMMING}
\end{center}

\bigskip

\begin{center}
VIVEK S.\ BORKAR\footnote{The author was supported in part by a S.\ S.\ Bhatnagar Fellowship from the Council for Scientific and Industrial Research, Government of India. He thanks Prof.\ K.\ S.\ Mallikarjuna Rao for bringing to his notice the work of Vrieze.}  \\
DEPARTMENT OF ELECTRICAL ENGINEERING,\\
INDIAN INSTITUTE OF TECHNOLOGY BOMBAY, \\
POWAI, MUMBAI 400076, INDIA, \\
Email: borkar.vs@gmail.com
\end{center}

\bigskip

\noindent \textbf{Abstract:} This article recalls the recent work on a linear programming formulation of infinite horizon risk-sensitive control via its equivalence with a single controller game,  using a classic work of Vrieze. This is then applied to a constrained risk-sensitive control problem with a risk-sensitive cost and risk-sensitive constraint. This facilitates a Lagrange multiplier based resolution thereof. In the process, this leads to an unconstrained linear program and its dual, parametrized by a parameter that is a surrogate for Lagrange multiplier. This also opens up the possibility of a primal - dual type numerical scheme wherein the linear program is a subroutine within the subgradient ascent based update rule for the Lagrange multiplier. This equivalent unconstrained risk-sensitive control formulation  does not seem obvious without the linear programming equivalents as intermediaries. We also discuss briefly other related algorithmic possibilities for future research.

\bigskip


\bigskip

\section{Introduction} Risk-sensitive control, which aims to minimize expectation of exponentiated sum (or integral if in continuous time) of per stage costs, was introduced in the finite time horizon framework by Howard and Matheson \cite{Howard} and was introduced to the control community by Jacobson \cite{Jac}. It received wide attention through the seminal work of Sir Peter Whittle \cite{Whittle1}, \cite{Whittle2}. The original motivation usually was to account for fluctuations, in view of the common criticism of expected cost or reward that it does not account for fluctuations around the average. Simple tweaks to fix this problem, such as considering a weighted sum of mean cost and  variance, lose the principle of time-consistency on which dynamic programming is based. That is, the segment  from some time $t>0$ onward of an optimal trajectory initiated at time zero may not remain optimal at time $t$. Risk-sensitive control on the contrary does account for fluctuations, being a strictly convex function of the cumulative running cost. At the same time it  allows for multiplicative dynamic programs that retain time consistency. It also has other selling points: it naturally captures the `compounding effect' which makes it attractive for finance applications \cite{Bielecki} and has close connections with robust control \cite{Basar}. See \cite{Biswas} for an extensive survey of this field.\\

Linear programming formulations of Markov decision processes go back to Howard \cite{How}, Manne \cite{Manne}, Kallenberg \cite{Kal}. For a fixed Markov policy (or \textit{stationary} Markov policy as appropriate), the dynamic programming equations for the classical cost criteria reduce to a system of linear equations that allow us to compute the cost functional under consideration for that particular policy, using the so called `one step analysis'. Thus in hindsight, existence of an equivalent linear program for the control problem should not come as a surprise. In fact, for classical costs/rewards, the control problems are `envelopes' of linear problems in a sense made precise by the Nisio semigroup \cite{Nisio}. For risk-sensitive control, fixing such a stationary Markov policy leads to an eigenvalue problem which is already nonlinear. Thus the validity of an equivalent linear program is not automatic. Fortunately, risk-sensitive control is equivalent to a zero sum game, a fact that was observed early in its history. What's more, it is equivalent to a \textit{single controller} zero sum game \cite{Filar} wherein only one of the two controllers controls the transition probabilities, the other can influence only the payoff. For this special case, an equivalent linear program was derived by Vrieze \cite{Vrieze}. This fact was exploited in order  to derive equivalent linear programs for the infinite horizon or `ergodic' risk-sensitive control in \cite{Ari}, \cite{B-CDC}. \\

The main contribution of this work is to the constrained risk-sensitive control problem, wherein one seeks to minimize a risk-sensitive cost with a bound on another risk-sensitive cost with a different per stage cost function. Constrained control problems for classical cost structures are themselves quite classical \cite{Derman}, \cite{Beutler} (see \cite{Altman} for a book length treatment). But because of the exponential nonlinearity applied to the cumulative running cost, the risk-sensitive case is not as straightforward as  the classical problems. However, the equivalent linear programs open a door to its resolution. Indeed, going via the equivalent linear programs, we are able to give a clean treatment of the constrained risk-sensitive control, \textit{en route} providing a  linear programming based computational scheme for it.\\

We summarize the developments in  \cite{Ari}, \cite{B-CDC} in the next section, which also takes this opportunity to fix some notational issues, typos and other small errors in \textit{ibid}. Section 3 is devoted to the aforementioned results on constrained risk-sensitive control. Section 4 concludes by pointing out future possibilities that include application of reinforcement learning to this problem. 

\section{Risk-sensitive control and its equivalent\\ linear programs}

Consider a controlled Markov chain $X_n, n \geq 0,$ on a finite state space $\SA$, controlled by a control process $U_n, n \geq 0,$ taking values in a finite action space $\A$, and evolving according to 
$$P(X_{n+1} = j|X_m, U_m, m \leq n) = p(j|X_n,U_n), \ n \geq 0,$$
for a prescribed controlled transition kernel $p(\cdot | \cdot , \cdot ) : \SA^2\times  \A \mapsto [0,1]$ satisfying $\sum_jp(j|i,u) = 1 \ \forall \ i,u$. We call such $\{U_n\}$ as \textit{admissible controls}. For a prescribed `running cost' $c(\cdot,\cdot): \SA\times \A \mapsto \RR^+$, consider the infinite horizon  risk-sensitive cost
\begin{equation}
\limsup_{N\uparrow\infty}\frac{1}{N}\Exp\left[e^{\sum_{m=0}^{N-1}c(X_m,U_m)}\big| X_0 = i\right]. \label{cost}
\end{equation} 
The infinite horizon risk-sensitive control problem for cost minimization seeks to minimize (\ref{cost}) over all admissible controls. A symmetric definition can be given for the risk-sensitive reward maximization problem by replacing `limsup' above by `liminf' and `cost minimization' by `reward maximization'. Recently, an equivalent linear programming formulation for the infinite horizon risk-sensitive control was given  in \cite{B-CDC}, \cite{Ari}, using its equivalence to a single controller game and using the linear programs for the latter derived by Vrieze \cite{Vrieze}. We  now briefly recall the key results of  \cite{Ari}.\\

 We do not assume irreducibility, whence
(\ref{cost}) can depend on the initial condition $i$. Recall that a control $\{U_n\}$ is called a stationary policy if $U_n = v(X_n)$ for all $n$ and some $v: \SA \mapsto \A$. By abuse of notation, it is often identified with this map $v(\cdot)$. It is called a stationary randomized policy if for each $n$, $U_n$ is conditionally independent of $X_m, U_m, m < n,$ given $X_n$. This is identified with the map $i \in \SA \mapsto\varphi(\cdot|i):=$ the conditional law of $U_n$ given $X_n = i$. We view $i \mapsto \varphi( \cdot |i)$ as a map $\SA\mapsto\PA(\A)$, where we denote by $\PA(\cdots)$ the simplex of probability vectors on the finite set `$\cdots$'. Denote by $\U_s, \U_{sr}$ the sets of stationary, resp.\ stationary randomized policies. Under a stationary policy (resp., stationary randomized policy) $U_n = v(X_n)$ (resp., $U_n \approx \varphi(\cdot|X_n)$)  $\forall \ n \geq 0$, for some $v: \SA \mapsto \A$ (resp., $\varphi: \SA \mapsto \PA(\A)$), (\ref{cost}) exists as a well defined limit, thanks to the `multiplicative ergodic theorem' for Markov chains \cite{Balaji}. We denote it  as $\lambda_i^v$ (resp., $\lambda_i^\varphi$) when $X_0 = i$. Let $\lambda^* := \max_i\min_v\lambda^v_i$. \\

Let $\Q :=$ the set of transition matrices $q = [[q(j|i)]]_{i,j\in  \SA}$ on $\SA$. Let $\Q(i) := \{q_i := q(\cdot|i), q \in \Q\}, i \in   \SA$, which is a copy of the simplex of probability vectors on $\SA$, indexed by $i\in  \SA$. Also define
$$\tilde{c}(i,u,q_i) := c(i,u) - D(q(\cdot|i)\|p(\cdot|i,u)).$$
Here
$$D(q(\cdot|i)\|p(\cdot|i,u)) := \sum_{j\in  \SA}q(j|i)\log\left(\frac{q(j|i)}{p(j|i,u)}\right)$$
if $q(\cdot|i) << p(\cdot|i,u)$, and  $= -\infty$ otherwise, is the Kullback-Leibler divergence. Consider a controlled Markov chain $\{\X_n\}$ on $\SA$ defined as follows. Its action space at $i\in  \SA$ is $\Q(i)\times  A$. The controlled transition probabilities are
$$\tilde{p}(j|i,q,u) = q(j|i), \ i,j\in  \SA.$$
Note the dual role of $q_i$, resp.\ as a transition probability and as a control.  The running \textit{payoff} for this controlled Markov chain  is $\tilde{c}(i,u,q_i)$ as above. \\

We shall consider only stationary policies $v: \SA \mapsto A$ for the state-dependent choice of the control variable $u$.  Let $\M_q$ denote the set of stationary distributions  for $q \in   \Q$. Define $\tilde{c}_v(i,q_i) := \tilde{c}(i,q_i,v(i))$. Then the risk-sensitive control problem above is equivalent to a zero sum stochastic game with payoff \cite{Ari}
\begin{equation}
\hat{\Phi}(q,v) = \max_{\pi\in  \M_q}\sum_{i\in  \SA}\pi(i)\tilde{c}_v(i,q_i). \label{phihat}
\end{equation}
This is a \textit{single controller game} \cite{Filar} in the sense that the transition probabilities are controlled by only one of the controllers (specifically, a fictitious entity that chooses the $q_i$), the other controller (the actual one, i.e., the one that chooses the $v$) controls only the payoff. It can be shown that this game has a value and it equals $\lambda^*$ \cite{Ari}. Furthermore, the following `min-max' theorem holds: \\

\begin{theorem}\label{minmax} \cite{Ari} $\hat{\Phi}^* := \inf_{v \in \U_s}\sup_{q \in \Q}\hat{\Phi}(q,v) = \sup_{q \in \Q}\inf_{v \in \U_s}\hat{\Phi}(q,v)$ and furthermore, there exist $v^* \in \U_s$ and $q^* \in \Q$ such that 
$$\hat{\Phi}^*= \inf_{v \in \U_s}\hat{\Phi}(q^*,v) = \sup_{q \in \Q}\hat{\Phi}(q, v^*) = \hat{\Phi}(q^*,v^*).$$
\end{theorem}
That is, there is a saddle point equilibriu for this zero sum game.\\

Then the linear program associated with the problem of minimizing (\ref{cost}) over all stationary policies can be derived as in \cite{Vrieze} and is given by \cite{Ari}:

\bigskip

\noindent (LP-P) Minimize $\sum_{i\in  \SA}\beta_i$ subject to:
\begin{eqnarray*}
\beta_i &\geq& \sum_{j\in  \SA}q(j|i)\beta_j, \ i\in  \SA,\\
V_i &\geq& \sum_{u\in   A}\tilde{c}(i,q_i,u)y_i(u) - \beta_i + \sum_{j\in  \SA}q(j|i)V_j, \ i \in   \SA,\\
y_i(u)&\geq& 0, \ \sum_jy_j(u) = 1.
\end{eqnarray*}

This is the `primal' linear program. The dual linear program is:

\bigskip

\noindent (LP-D) Maximize $\sum_iw_i$ subject to:
\begin{eqnarray}
\sum_{i\in  \SA}\int_{\Q}(\delta_{ij} - q(j|i))\mu(i,dq) &=& 0, \ j\in  \SA, \nonumber\\
\sum_{i\in  \SA}\int_{\Q}(\delta_{ij} - q(j|i))\nu(i,dq) + \int_{\Q}\mu(j,dq) &=& 1, \ j\in  \SA, \nonumber\\
\int_{\Q}\tilde{c}(j,u,q_j)\mu(j,dq) &\geq& w_j, \ j\in  \SA, \label{special} \\
\sum_j\int\mu(j,dq') = 1, \ \mu(i,q),\ \nu(i,q) &\geq& 0, \ (i,q)\in  \SA\times  \Q. \nonumber
\end{eqnarray}
Here $\delta_{ij}$ denotes the Kronecker delta.\\

The work of \cite{Vrieze} that is cited here is for finite state and action spaces. The action spaces $\Q(i)$ above are not finite, so one has to go through a sequence of finite approximations thereof, followed by a limiting argument. See \cite{Ari},  \cite{B-CDC} for details.  One can show that these linear programs are feasible and have bounded solutions, there is no duality gap,  and the optimal solution is precisely the value of the zero sum game defined in the preceding section. Furthermore, the optimal stationary randomized policy is optimal among all admissible policies and can be recovered from the solution of the primal program as $i \mapsto y_i(\cdot)$. (We omit the details, referring the reader to \cite{Ari}). \\

One can reverse engineer the dynamic programming equations from this. These are as follows \cite{Ari}.

\medskip

\begin{theorem}
The dynamic programming equation, given by:
\begin{eqnarray*}
\Psi_i  &=& \max_{q_i\in  \Q(i)}\sum_jq_{ij}\Psi_j, \\
\Psi_i + V_i &=& \min_{u\in   A}\max_{q_i\in   B_i}\left[\tilde{c}(i,u,q_i) + \sum_jq(j|i)V_j\right], 
\end{eqnarray*}
where
\begin{eqnarray*}
B_i &:=& \left\{q_i \in   \Q(i) : \sum_j q(j|i)\Psi_j = \Psi_i\right\},
\end{eqnarray*}
has a solution $\{(\Psi_i, V_i)\}_{i\in  \SA}$. We also have $\Psi_i = \beta_i \ \forall\  i\in  \SA$ where $\{\beta_i\}$ is the solution to (LP) and $=$ the value of the above zero sum game with initial condition $i$, $\forall i$. Furthermore, $\lambda^* = \max_i\Psi_i$. 
\end{theorem}

The state space $\SA$ can be partitioned into disjoint subsets as $\SA = \cup_{\ell = 1}^kI_\ell$ where $I_j := \{j : \Psi_i = \beta_j, \  \forall \ i\in   I_j\}$, $1 \leq j \leq k$.
 Let $\Lambda_i := e^{\lambda_i} \ \forall i$ and $\Lambda^* := e^{\lambda^*}$. Performing the maximization with respect to $q_i \in  \Q(i)$  exactly, one can write the dynamic programming equation as
\begin{eqnarray*}
\Lambda^* &=& \max_i\Lambda_i, \\
\Lambda_i\zeta_i &=& \min_u\left(\sum_{j\in   I_i}p(j|i,u)e^{c(i,u)}\zeta_j\right), \ i \in   I_\ell, 1 \leq \ell \leq k, \\
\Lambda_i &=& \min_{B_i^*}\sum_{j\in   I_i}\left(\frac{p(j|i,u)e^{c(i,u)}\zeta_j}{\sum_{j'}p(j'|i,u)e^{c(i,u)}\zeta_{j'}}\right)\Lambda_{j'}, \ (i,j) \in   I_\ell, 1 \leq \ell \leq k,
\end{eqnarray*}
where $\B^*_i$ is the set of minimizers in the second equation above. This is the counterpart for risk-sensitive control of the classical result of Kallenberg \cite{Kal} for average cost dynamic programming equation for multi-chain problems.
Analogous results are possible  for risk-sensitive reward problems, see \cite{Ari}, \cite{B-CDC}.

\section{Risk-sensitive control with constraints}

As a spin-off of the foregoing development, we can handle risk-sensitive control with risk-sensitive constraints. Consider minimization of (\ref{cost}) subject to an additional constraint
\begin{equation}\label{constraint}
\limsup_{N\to\infty}\frac{1}{N}\log\Exp\left[e^{\sum_{t=0}^{N-1}k(X_m,U_m)}\right] \leq C, 
\end{equation}
for a prescribed $k: S\times  \A \mapsto \RR$ and a constant $C > 0$. Define
$$\tilde{k}(i,u,q_i) := k(i,u) - D(q(\cdot|i)\|p(\cdot|i,u)).$$  Denote the $\widehat\Phi(q,v)$ defined in (\ref{phihat}) above as $\widehat\Phi(q,v,c)$ in order to render explicit its dependence on the per stage cost function $c$. We can also define $\widehat\Phi(q,v,k)$ analogously. This leads to the convex program (over the set of mixed policies for the agent choosing $v$, i.e., over probability distributions on $\U_{s}$) given by:\\

\noindent Minimize $\max_{q\in  \mathcal{Q}}\widehat\Phi(q,v,c)$ subject to $\max_{q'\in  \mathcal{Q}}\widehat\Phi(q',v,k) \leq C.$\\

Then by standard Lagrange multiplier theory (see, e.g., \cite{Luen}), one can consider the unconstrained minimization of
$$\max_{q\in  \mathcal{Q}}\widehat\Phi(q,v,c) + \Gamma\left(\max_{q'\in  \mathcal{Q}}\widehat\Phi(q',v,k) - C\right),$$
where $\Gamma \geq 0$ is a Lagrange multiplier. We now write down the equivalent primal and dual linear programs and justify them later. We treat $\Gamma$ as a given parameter which parametrizes the primal and dual linear programs. \\

The primal linear program is:\\

\noindent  (LP-P$^*$) Minimize $\sum_{i\in   S}(\beta_i + \Gamma(\beta'_i - C))$ subject to:

\begin{align}
\beta_i &\geq \sum_{j\in   S}q(j|i)\beta_j, \ (i,q)\in  S\times \mathcal{Q}, \label{P1}
\\
V_i &\geq \sum_{u\in  U}\tilde{c}(i,q,u)y_i(u) - \beta_i + \sum_{j\in  S}q(j|i)V_j, \ (i,q) \in   S\times \mathcal{Q}, \label{P2}
\\
\beta_i' &\geq \sum_{j\in  S}q'(j|i)\beta_j', \ (i,q')\in  S\times \mathcal{Q}, \label{P3}
\\
V_i' &\geq \sum_{u\in  U}\tilde{k}(i,q',u)y_i(u) - \beta_i' + \sum_{j\in  S}q'(j|i)V_j', \ (i,q') \in   S\times \mathcal{Q}, \label{P4}
\\
y_i(u)&\geq 0, \ \sum_j y_j(u) = 1. \label{P5}
\end{align}

The dual linear program is: \\

\noindent(LP-D$^*$) Maximize $\sum_{i,j\in   S}w_{i}$ subject to:

\begin{align}
\sum_i\int_{\Q}
(\delta_{ij} - q(j|i))\mu(i,dq) &= 0, \ j\in  S,\label{D1}
\\
\sum_i\int_{\Q}(\delta_{ij} - q(j|i))\nu(i,dq) + \int_{\Q}\mu(j,dq') &= 1, \ j\in  S, \label{D2} \\
\sum_I\int_{\Q}
(\delta_{ij} - q'(j|i))\mu'(i,dq') &= 0, \ j\in  S,\label{D3}
\\
\sum_i\int_{\Q}(\delta_{ij} - q'(j|i))\nu'(i,dq') + \int_{q' \in \Q}\mu'(j,dq') &= 1, \ j\in  S, \label{D4} \\
\int_{\Q}\tilde{c}(i,q,u)\mu(i,dq) + \Gamma\left(\int_{\Q(i)}\tilde{k}(i,q',u)\mu'(i,dq') - C\right) &\geq w_{i}, \ (i,u)\in  S\times U, \label{D5}\\
\sum_i\int\mu(i,d\tilde{q} = \sum_i\int\mu'(j,d\tilde{q}) &= 1, \label{D6} \\
\mu(i,u), \mu'(i,u), \nu(i,u), \nu'(i,u)& \geq 0, \ (i,u)\in S\times\Q.    \label{D7}
\end{align}

\bigskip

Here the Lagrange multiplier $\Gamma$ is unknown a priori. It will be kept fixed in the statement and proof of the next theorem. We return to this point later.

\bigskip

\begin{theorem} The primal and dual linear programs (LP-P$^*$) and  (LP-D$^*$) above are the linear programming formulations of the constrained risk-sensitive control problem.\end{theorem}

\begin{proof} We do this in two steps.

\medskip

\noindent \textbf{Proof of the validity of (LP-P$^*$) :}

\medskip

Note that the pair  (\ref{P1})-(\ref{P2}) and the pair (\ref{P3})-(\ref{P4}) are coupled only through the variables $\{y_i(u)\}$ from (\ref{P5}). Thus if we fix $\{y_i(u)\}$, which is tantamount to freezing the randomized stationary policy of the player who controls only the payoff, the constraint pairs (\ref{P1})-(\ref{P2}) and (\ref{P3})-(\ref{P4}) decouple. The objective function being minimized is separable in the variables $\{\beta_i\}$ and $\{\beta_i'\}$. Thus the problem decouples into two \textit{single agent} control problems of resp., minimizing $\sum_i\beta_i$ subject to the constraints (\ref{P1})-(\ref{P2}), and minimizing $\sum_i\beta_i'$ subject to the constraints (\ref{P3})-(\ref{P4}). Both are precisely of the form of Kallenberg's linear program for the multi-chain case of average reward control \cite{Kal}. Therefore the respective optimal rewards are precisely the maxima over $q$, $q'$ of the respective objectives. Furthermore, by the results of \cite{Ananth}, these are precisely the  risk-sensitive rewards for running cost functions $c(\cdot,\cdot), k(\cdot,\cdot)$ resp., and a fixed stationary randomized policy  $y_\cdot(\cdot)$.

Next, allow $\{y_i(u)\}$ to vary, i.e., we require that the constraints hold for all choices of $\{y_i(u)\}$ subject to (\ref{P5}). This is equivalent to minimizing over $\{y_i(u)\}$ the maximum over $q, q'$ obtained above. In view of the min-max theorem (Theorem \ref{minmax}), we are done. 

\medskip

\noindent \textbf{Proof of the validity of (LP-D$^*$) :}

\medskip
 
Now the pair (\ref{D1})-(\ref{D2}) and the pair (\ref{D3})-(\ref{D4}) are coupled only through the constraint (\ref{D5}). Fix $u$. If we ignore (\ref{D5}), then (\ref{D1})-(\ref{D2}) along with the domains of $\mu, \nu$ specified as in (\ref{D6}), (\ref{D7}), is the feasible set for Kallenberg's dual linear program for multi-chain average reward Markov decision processes \cite{Kal}, where we take the average reward to be maximized to correspond to the per stage running reward function $\tilde{c}(i, q, u)$ with $u$ fixed. Likewise, the constraints (\ref{D3})-(\ref{D4}) correspond to the dual linear program for the average reward Markov decision process with running reward $\tilde{k}(i,q,u), i \in S, q \in \Q$, with $u \in \A$ kept fixed. The constraint (\ref{D5}) is required to hold for all $u$. This then forces the right hand side of (\ref{D5}) to be dominated by the minimum over $u$ of its left hand side. This, in view of the objective function $\sum_iw_i$ of (LP-D$^*$), achieves the minimization over $u$ of the payoff maximized over $q$. \\

This completes the proof. \end{proof}

\bigskip

This allows us to map the constrained risk-sensitive control problem into an unconstrained one just as in the case of classical cost/reward structures. For this purpose, define an $\SA^2$-valued Markov chain $Y_n = (Z_n, Z'_n), n \geq 0$, as follows. It is initiated at some $(i,i') \in \SA^2$ with transition probabilities $\check{p}((k,\ell)|(i,j), \check{q}, u) := \check{q}((k,\ell)|(i,j))$ given by
$$\check{q}((k,\ell)|(i,j)) := q(k|i)q'(\ell|j), \ i,j,k,\ell \in \SA.$$
In particular, the transitions are independent. The per stage cost is
$$\check{h}((i,j),u) := c(i,u) + \Gamma k(j,u).$$
The objective is to minimize  the risk-sensitive cost
\begin{equation}
\limsup_{N\uparrow\infty}\frac{1}{N}\log E\left[e^{\sum_{n=0}^{N-1}\check{h}(Y_n,\check{U}_n)}\right]. \label{cost2}
\end{equation}

\begin{theorem} The constrained risk-sensitive control problem under consideration is equivalent to the above unconstrained risk-sensitive problem in the sense that they have the same optima. \end{theorem}

\begin{proof} Define $\check{V}_i := V_i + \Gamma V_i', \check{\beta}(i,j) := \beta_i + \Gamma \beta_j' \ \forall \ i,j$. Then summing the l.h.s.\ (resp., the r.h.s.) of (\ref{P1}) and $\Gamma\times$ the l.h.s.\ (resp., the r.h.s.) of (\ref{P3}), we get
\begin{equation}
\check{\beta}_{(i,j)} \geq \sum_{(k,\ell)}\check{q}((k,\ell)|(i,j))\check{\beta}_{(k,\ell)}. \label{C1}
\end{equation}
Likewise, summing the l.h.s.\ (resp., the r.h.s.) of (\ref{P2}) and $\Gamma\times$ the l.h.s.\ (resp., the r.h.s.) of (\ref{P4}), we get
\begin{equation}
\check{V}_{(i,j)} \geq \sum_{u\in \A}\tilde{h}((i,j),\check{q},u)y_i(u) - \check{\beta}_{(i,j)}
+ \sum_{(k,\ell) \in \SA^2}\check{q}((k,\ell)|(i,j))\check{V}_{(k,\ell)}. \label{C2}
\end{equation}
Consider the linear program that seeks to minimize $\sum_{(i,j) \in \SA^2}\check{\beta}_{(i,j)}$ subject to constraints (\ref{C1}), (\ref{C2}) and (\ref{P5}). This is precisely the linear program corresponding to the undiscounted risk-sensitive control problem of minimizing 
(\ref{cost2}). Since both this and the original problem attain the same optimum, the claim follows.
\end{proof}

In particular, this allows us to solve an equivalent unconstrained problem. The problem is that this pre-supposes the knowledge of $\Gamma$.
In practice, $\Gamma$ too needs to be computed. 
One can use the following `primal-dual' scheme. Start with an initial guess for $\Gamma$, say $\Gamma_0 \geq 0$, and update it as follows. At step $n\geq 0$, solve the above linear programs for $\Gamma = \Gamma_n$. Let $\mu^n(\cdot, \cdot)$ be the optimal $\mu'(\cdot, \cdot)$ from the dual linear program (LP-D$^*$) and $y^n_\cdot(\cdot)$  the optimal $y_\cdot(\cdot)$ for the primal linear program (LP-P$^*$), under $\Gamma = \Gamma_n$. Perform the iterate
\begin{equation}
\Gamma_{n+1} = \left[\Gamma_n + a(n)\left(\int_{\Q(i)}\sum_u\tilde{k}(i,q',u)y_i^n(u)d\mu^n(i,dq') - \log C\right)\right]^+, \label{ascent}
\end{equation}
where $a(n) > 0$ is a stepsize sequence satisfying $a(n)\to 0, \ \sum_na(n) = \infty$, and $[\cdots]^+$ denotes the projection to the positive real axis. Letting 
$$\Psi(\Gamma) := \min_v\left(\max_{q\in\Q}\hat{\Phi}(q,v,c) + \Gamma\left(\max_{q'\in\Q}\hat{\Phi}(q',v,k) - C\right)\right),$$
we note that this is a concave function of $\Gamma$ on $[0,\infty)$. From the standard Lagrange multiplier theory \cite{Luen}, the Lagrange multiplier is precisely a maximizer of this function. Using Danskin's theorem \cite{Danskin}, (\ref{ascent}) is then seen to be  the projected subgradient ascent for finding the Lagrange multiplier. By standard theory of subgradient ascent, it will a.s.\ converge to a maximizer of $\Psi(\cdot)$, i.e., a Lagrange multiplier.

\section{Reinforcement learning}

In the preceding section, the linear programs were treated as a subroutine within the subgradient ascent for the Lagrange multiplier. While they were useful intermediaries to derive the equivalent unconstrained risk-sensitive control problem with cost (\ref{cost2}), now that this has been done, one can fall back on other classical iterative schemes for risk-sensitive control such as policy iteration and value iteration (see \cite{Biswas} for a survey). They can in fact be run on a faster time scale to emulate a subroutine, using the standard `two time scale' paradigm. That is, we run the two iterations (i.e., value/policy iteration and the subgradient ascent) with different stepsizes (viewed as small time steps), say, $\{a(n)\}$ and $\{b(n)\}$, with the usual requirements that $a(n), b(n) \to 0$ and $\sum_na(n) = \sum_nb(n) = \infty$, along with the additional requirement that  $b(n) = o(a(n))$. The latter condition ensures that the former iteration runs on a faster algorithmic time scale and sees the latter iteration as quasi-static, while the latter running on a slower time scale  sees the former as quasi-equilibrated so that it emulates a subroutine. This is standard fare in stochastic approximation and has been extensively studied, see, e.g., Chapter 8 of \cite{BorkarBook}. The additional condition $\sum_n(a(n)^2+b(n)^2) < \infty$ required in stochastic approximation in order to contain the effect of noise is, however, not needed here. It is needed though in the stochasic counterparts of the above, viz., reinforcement learning, which we discuss next.

 An immediate spin-off of the foregoing is the fact that the existing reinforcement learning algorithms for risk-sensitive cost, such as Q-learning \cite{Q}, \cite{BBB}, actor-critic \cite{A-C} or policy gradient \cite{Mohr} can be immediately adapted to the constrained problem by using them for the equivalent unconstrained risk-sensitive control problem with cost (\ref{cost2}), and then appending a stochastic subgradient ascent for learning the Lagrange multiplier on a slower time scale,  given by
$$\Gamma(n+1) = \Gamma(n) + b(n)(k(X_n,U_n) - \log C), \ n \geq 0,$$
where the stepsize $\{b(n)\}$ satisfies the usual Robbins-Monro conditions: 
$$\sum_n b(n) = \infty, \ \sum_n b^2(n) < \infty,$$ 
along with the additional condition $b(n) = o(a(n))$ for any stepsize sequence $\{a(n)\}$ being used in the (algorithm-dependent) primal iteration. (Note that $\{a(n)\}$ is required to satisfy the Robbins-Monro conditions too.) This ensures that the Lagrange multiplier update is on a slower time scale, which facilitates the application of the theory of two time scale analysis. We omit the details as they are fairly routine (\cite{BorkarBook}, Chapter 8).

\end{document}